\documentclass[11pt]{birkjour}
\usepackage{palatino}
\usepackage{amsfonts}
\usepackage{amssymb}
\usepackage{amscd}
\usepackage{xypic}

\renewcommand{\labelenumi}{(\roman{enumi})}

\newtheorem{theorem}{Theorem}[section]
\newtheorem{proposition}[theorem]{Proposition}

\theoremstyle{definition}
\newtheorem{definition}[theorem]{Definition}
\newtheorem{remark}[theorem]{Remark}
\newtheorem{example}[theorem]{Example}

\newtheorem{conjecture}[theorem]{Conjecture}
\newtheorem{conjecture/question}[theorem]{Conjecture/Question}
\newtheorem{question}[theorem]{Question}
\newtheorem{remark/definition}[theorem]{Remark/Definition}
\newtheorem{terminology/notation}[theorem]{Terminology/Notation}

\newcommand{\dblq}{{/\!/}}

\def\PP{{\textbf P}}
\def\FF{{\textbf F}}
\def\OO{\mathcal{O}}

\def\cA{\mathcal{A}}

\def\cS{\mathcal{S}}

\def\L{\mathcal{L}}

\def\cM{\mathcal{M}}
\def\cR{\mathcal{R}}

\def\cZ{\mathcal{Z}}
\def\cU{\mathcal{U}}

\def\Pic0{{\rm Pic}^0(X)}

\def\mm{\overline{\mathcal{M}}}
\def\ss{\overline{\mathcal{S}}}
\def\cc{\overline{\mathcal{C}}}

\def\zz{\overline{\cZ}}

\def\thet{\overline{\Theta}_{\mathrm{null}}}

\begin{document}
\title{Theta characteristics and their moduli}

\author[G. Farkas]{Gavril Farkas}

\address{Humboldt-Universit\"at zu Berlin, Institut F\"ur Mathematik,  Unter den Linden 6, \hfill \
 Berlin 10099, Germany} \email{{\tt farkas@math.hu-berlin.de}}
\thanks{}

\begin{abstract}
We discuss topics related to the geometry of theta characteristics on algebraic curves. They include the birational classification of the moduli space $\mathcal{S}_g$ of spin curves of genus $g$, interpretation of theta characteristics as quadrics in a vector space over the field with two elements as well as the connection with modular forms and superstring scattering amplitudes. Special attention is paid to the historical development of the subject.
\end{abstract}

\maketitle
Theta characteristics appeared for the first time in the context of characteristic theory of odd and even theta functions in the
papers of G\"opel \cite{Go} and Rosenhain \cite{Ro} on Jacobi's inversion formula for genus $2$. They were initially considered in connection with
\emph{Riemann's bilinear addition} relation between the degree two monomials in theta functions with characteristics. Later, in order to systematize the relations between theta constants, Frobenius \cite{Fr1}, \cite{Fr2} developed an algebra of characteristics \footnote{Frobenius' attempts to bring algebra into the theory of theta functions has to be seen in relation to his famous work on group characters. In 1893, when entering the Berlin Academy of Sciences he summarized his aims as follows \cite{Fr3}: \emph{In the theory of theta functions it is easy to set up an arbitrarily large   number of relations, but the difficulty begins when it comes to finding a way out of this labyrinth of formulas. Many a distinguished researcher, who through tenacious perseverence, has advanced the theory of theta functions in two, three, or four variables, has, after an outstanding demonstration of brilliant talent, grown silent either for a long time of forever. I have attempted to overcome this paralysis of the mathematical creative powers, by seeking renewal at the fountain of youth of arithmetic}.}; he distinguished between \emph{period} and \emph{theta} characteristics. The distinction, which in modern terms amounts to the difference between the \emph{Prym} moduli space
$\cR_g$ and the \emph{spin} moduli space $\cS_g$, played a crucial role in elucidating the transformation law for theta functions under a linear
transformation of the moduli, and it ultimately led to a correct definition of the action of symplectic group $\mbox{Sp}(\FF_2^{2g})$ on the set of characteristics. An overwiew of the 19th century theory of theta functions can be found in Krazer's monumental book \cite{Kr}. It is a very analytic treatise in character, with most geometric applications either completely absent relegated to footnotes.
\vskip 5pt

The remarkable book \cite{Cob} by Coble
\footnote{An irreverent portrait of Coble in the 1930's from someone who was not exactly well disposed towards algebraic geometry ("\emph{... the only part of mathematics where a counterexample to a theorem is considered to be a beautiful addition to it}"), can be found in Halmos' autobiography "I want to be a mathematician".}
represents a departure from the analytic view towards a more abstract understanding of theta characteristics using configurations in finite geometry. Coble viewed  theta characteristics as quadrics in a vector space over $\FF_2$. In this language Frobenius' earlier concepts (syzygetic and azygetic triples, fundamental systems of characteristics) have an elegant translation. With fashions in algebraic geometry drastically changing, the work of Coble was forgotten for many decades \footnote{This quote from Mattuck's \cite{Ma} obituary of Coble reveals the pervasive attitude of the 1960's: \emph{The book as a whole is a difficult mixture of algebra and analysis, with intricate geometric reasoning of a type few can follow today. The calculations are formidable; let them serve to our present day algebraic geometers, dwelling as they do in their Arcadias of abstraction, as a reminder of what awaits those who dare to ask specific questions about particular varieties}.}.
\vskip 3pt

The modern theory of theta characteristic begins with the works of Atiyah \cite{At} and Mumford \cite{Mu}; they showed, in the analytic (respectively algebraic) category, that the parity of a theta characteristic is stable under deformations. In particular, Mumford's functorial view of the subject, opened up the way to extending the study of theta characteristics to singular curves (achieved by Harris \cite{H}), to constructing a proper Deligne-Mumford moduli space of stable spin curves (carried out by  Cornalba \cite{C}), or to reinterpreting Coble's work in modern terms (see the book \cite{DO}).

\vskip 3pt
The aim of this paper is to survey various developments concerning the geometry of moduli spaces of spin curves. Particular emphasis is placed on the complete birational classification of both the even and the odd spin moduli spaces, which has been carried out in the papers \cite{F2}, \cite{FV1} and \cite{FV2}. Precisely, we shall explain the following result:
\begin{theorem} The birational type of the moduli spaces $\ss_g^+$ and $\ss_g^-$ of even and odd spin curves of genus $g$  can be summarized as follows:
\begin{center}
$\ss_g^+$: \ \begin{tabular}{c|c}
 $g> 8$ & general type\\
 \hline
 $g=8$  &Calabi-Yau \\
 \hline
$g\leq 7$ & unirational
\end{tabular} \ \ \ \mbox{ } \mbox{  }
$\ss_g^-$: \ \begin{tabular}{c|c}
 $g\geq 12$ & general type\\
 \hline
  $g\leq 8$  & unirational\\
 \hline
 $9\leq g\leq 11$ & uniruled
\end{tabular}
\end{center}
\end{theorem}
We describe the structure of the paper. In the first two sections we recall the interpretation of theta characteristics as quadrics in an $\FF_2$-vector space, then link this description both with the classical theory of characteristics of theta functions and modern developments inspired by string theory. In the next three sections we explain some features of the geometry of the moduli space $\ss_g$ of stable spin curves of genus $g$, discuss ways of constructing effective divisors on $\ss_g$ and computing their cohomology classes, then finally present unirational parametrizations of the moduli space in small genus, by using Mukai models and special $K3$ surfaces. We close by surveying a few open problems related to syzygies of theta characteristics and stratifications of the moduli space.

\section{Theta characteristics: a view using finite geometry}
For a smooth algebraic curve $C$ of genus $g$ we denote by $$J_2(C):=\{\eta\in \mbox{Pic}^0(C):\eta^{\otimes 2}=\OO_C\}$$ the space of two-torsion points in the Jacobian of $C$, viewed as an $\FF_2$-vector space. We recall the definition of the \emph{Weil pairing} $\langle \cdot, \cdot \rangle:J_2(C)\times J_2(C)\rightarrow \FF_2$, cf. \cite{Mu} Lemma 2:
\begin{definition} Let $\eta, \epsilon\in J_2(C)$ and write $\eta=\OO_C(D)$ and $\epsilon=\OO_C(E)$, for certain divisors $D$ and $E$ on $C$, such that $\mbox{supp}(D)\cap \mbox{supp}(E)=\emptyset$. Pick rational functions $f$ and $g$ on $C$, such that $\mbox{div}(f)=2D$ and $\mbox{div}(g)=2E$. Then define $\langle \eta, \epsilon \rangle\in \FF_2$ by the formula
$$(-1)^{\langle \eta, \epsilon \rangle}=\frac{f(E)}{g(D)}.$$
\end{definition}
The definition of $\langle \eta, \epsilon\rangle$ is independent of the choice of the divisors $D$ and $E$ and the rational functions $f$ and $g$. The Weil pairing is a nondegenerate symplectic form.
The set of \emph{theta characteristics} of $C$, defined as $\mbox{Th}(C):=\{\theta\in \mbox{Pic}^{g-1}(C): \theta^{\otimes 2}=K_C\}$, is an affine space over $J_2(C)$. It was Coble's insight \cite{Cob} to realize that in order to acquire an abstract understanding of the geometry of $\mbox{Th}(C)$ and clarify the distinction between the period characteristics and the theta characteristics of $C$, it is advantageous to view theta characteristics as quadrics in the vector space $J_2(C)$.
\begin{definition}
Let $(V, \langle \cdot, \cdot \rangle)$ be a symplectic vector space over $\FF_2$. The set $Q(V)$ of quadratic forms on $V$ with fixed polarity
given by the symplectic form $\langle \cdot, \cdot \rangle$, consists of all functions $q:V\rightarrow \FF_2$ satisfying the identity
$$q(x+y)=q(x)+q(y)+\langle x, y\rangle, \ \mbox{ for all } x, y\in V.$$
\end{definition}
If $q\in Q(V)$ is a quadratic form and $v\in V$, we can define a new quadratic form $q+v\in Q(V)$ by setting
$(q+v)(x):=q(x)+\langle v, x\rangle$,  for all  $x\in V.$ Similarly one can add two quadratic forms. If $q, q'\in Q(V)$, then there exists a uniquely determined element $v\in V$ such that $q'=q+v$, and we set $q+q':=v\in V$. In this way, the set $\widetilde{V}:=V\cup Q(V)$ becomes a $(2g+1)$-dimensional vector space over $\FF_2$. There is a natural action of the symplectic group $\mathrm{Sp}(V)$ on $Q(V)$. For a transformation
$T\in \mathrm{Sp}(V)$ and a quadratic form $q\in Q(V)$, one defines
$$(T\cdot q)(x):=q(T^{-1}(x)), \mbox{ for } \ x\in V.$$
This action has two orbits $Q(V)^+$ and $Q(V)^-$ respectively, which can de distinguished by the \emph{Arf invariant} \cite{Arf} of a quadratic form.
\begin{definition} Let us choose a symplectic base $(e_1, \ldots, e_g, f_1, \ldots, f_g)$ of $V$. Then define
$$\mathrm{arf}(q):=\sum_{i=1}^g q(e_i)\cdot q(f_i)\in \FF_2.$$
\end{definition}
The invariant $\mbox{arf}(q)$ is independent of the choice of a symplectic basis of $V$. We set $Q(V)^+:=\{q\in Q(V): \mbox{arf}(q)=0\}$ to be the space of even quadratic form, and $Q(V)^-:=\{q\in Q(V): \mbox{arf}(q)=1\}$ that of odd quadratic forms.
Note that there are $2^{g-1}(2^g+1)$ even quadratic forms and $2^{g-1}(2^g-1)$ odd ones.

\vskip 3pt
Every theta characteristic $\theta\in \mbox{Th}(C)$ defines a form $q_{\theta}:J_2(C)\rightarrow \FF_2$, by setting
$$q_{\theta}(\eta):=h^0(C, \eta\otimes \theta)+h^0(C, \theta) \ \mathrm{ mod }\ 2.$$
It follows from the \emph{Riemann-Mumford relation} \cite{Mu} or \cite{H} Theorem 1.13, that $q_{\theta}\in Q(J_2(C))$, that is, the polar of $q_{\theta}$ is the Weil form. For $\eta, \epsilon\in J_2(C)$, the following relation holds:
$$h^0(C, \theta\otimes \eta\otimes \epsilon)+h^0(C, \theta\otimes \eta)+h^0(C, \theta\otimes \epsilon)+h^0(C, \theta)\equiv \langle \eta, \epsilon\rangle\ \mbox{ mod }\ 2.$$
Thus one has the following identification between theta characteristics and quadrics:
\begin{center}
\fbox{$\mathrm{Th}(C)=Q\bigl(J_2(C), \ \langle \cdot, \cdot \rangle \bigr).$}
\end{center}
Under this isomorphism, even (respectively odd) theta characteristics correspond to forms in $Q(V)^+$ (respectively $Q(V)^-$). Furthermore, $\mbox{arf}(q_{\theta})=h^0(C, \theta) \mbox{ mod } 2$.
Using this identification, one can translate Frobenius's
\cite{Fr1}, \cite{Fr2} entire \emph{theory of fundamental systems} of theta characteristics into an abstract setting. Of great importance is the following:
\begin{definition}
A system of three theta characteristics $\theta_1, \theta_2, \theta_3\in \mathrm{Th}(C)$ is called \emph{syzygetic} (respectively
\emph{azygetic}) if $\mbox{arf}(q_{\theta_1})+\mbox{arf}(q_{\theta_2})+\mbox{arf}(q_{\theta_3})+\mbox{arf}(q_{\theta_1+\theta_2+\theta_3})=0$ (respectively $1$).
\end{definition}
Here the additive notation $\theta_1+\theta_2+\theta_3$ refers to addition in the extended vector space $\widetilde{J_2(C)}=J_2(C)\cup \mbox{Th}(C)$. In terms of line bundles,  $\theta_1+\theta_2+\theta_3=\theta_1\otimes \theta_2\otimes \theta_3^{\vee}\in \mbox{Pic}^{g-1}(C)$. It is an easy exercise to show that $\{\theta_1, \theta_2, \theta_3\}$ is a syzygetic triple if and only if $\langle \theta_1+\theta_2, \theta_1+\theta_3\rangle=0$. If the system $\{\theta_1, \theta_2, \theta_3\}$ is syzygetic, then any three elements of the set $\{\theta_1, \theta_2, \theta_3, \theta_1+\theta_2+\theta_3\}$ form a syzygetic triple. In this case we say that the four theta characteristics form a \emph{syzygetic tetrad}.

\begin{example} Four odd theta characteristics $\theta_1, \ldots, \theta_4\in \mbox{Th}(C)$ corresponding to contact contact divisors $D_i \in |\theta_i|$  such that $2D_i\in |K_C|$ form a syzygetic tetrad, if and only if there exists a quadric $Q\in \mbox{Sym}^2 H^0(C, K_C)$, such that
$$Q\cdot C=D_1+D_2+D_3+D_4.$$
For instance, when $g=3$, four bitangents to a quartic $C\subset \PP^2$ form a syzygetic tetrad exactly when the $8$ points of tangency are the complete intersection of $C$ with a conic.
\end{example}
To understand the geometry of the configuration $\mbox{Th}(C)$ one defines appropriate systems of coordinates. Following \cite{Kr} p. 283,
one says that $2g+2$ theta characteristics $\{\theta_1, \ldots, \theta_{2g+2}\}$ form a \emph{fundamental system} if any triple $\{\theta_i, \theta_j, \theta_j\}$ where $1\leq i<j<k\leq 2g+2$ is azygetic. The sum of the $2g+2$ elements of a fundamental system equals $0$, see \cite{Kr} p. 284, and the number of fundamental systems in $\mbox{Th(C)}$ is known, cf. \cite{Kr} p. 285:
$$2^{2g}\frac{|\mathrm{Sp}(\FF_2^{2g})|}{(2g+2)!}=2^{2g}\frac{(2^{2g}-1)(2^{2g-2}-1)\cdots (2^2-1)}{(2g+2)!}2^{g^2}.$$

\begin{example}
A smooth plane quartic $C\subset \PP^2$ has precisely $288$ fundamental systems $(\theta_1, \ldots, \theta_8)$, where the first $7$ theta characteristics are odd, and $\theta_8=-\sum_{i=1}^7 \theta_i$ is then necessarily even. All elements in $\mbox{Th}(C)$ can be expressed in the "coordinate system" given by $(\theta_1, \ldots, \theta_7)$: The remaining odd theta characteristics are  $\theta_i+ \theta_j+ \theta_k+ \theta_l+ \theta_m$. The $35=36-1$ even theta characteristics are of the form $\theta_i+ \theta_j+ \theta_k$.
\end{example}
One can also consider systems of syzygetic theta characteristics, that is subsets, $\{\theta_1, \ldots, \theta_{r+1}\}\subset \mbox{Th}(C)$ such that all
triples $\{\theta_i, \theta_j, \theta_k\}$ are syzygetic. Then $r\leq g$, see \cite{Kr} p. 299. Following Frobenius \cite{Fr1}, a maximal system of such theta characteristics is called a \emph{G\"opel system} and corresponds to $2^g$ theta characteristics such any three of them form a syzygetic triple. These definitions can be immediately extended to cover general principally polarized abelian varieties not only Jacobians.
\section{Theta characteristics: the classical view via theta functions}
We now link the realization  of theta characteristics in abstract finite geometry, to the theory of theta functions with characteristics. There are established classical references, above all \cite{Kr}, \cite{Wi}, \cite{Ba}, \cite{Cob}, as well as modern ones, for instance \cite{BL}, \cite{SM}. We fix an integer $g\geq 1$ and denote by
$$\mathfrak{H}_g:=\{\tau\in M_{g, g}(\mathbb C): \tau=^t\tau, \ \mbox{Im }\tau>0\}$$ the Siegel upper half-space
of period matrices for abelian varieties of dimension $g$; hence $\cA_g:=\mathfrak{H}_g/\mathrm{Sp}_{2g}(\mathbb Z)$ is the moduli space of principally polarized abelian varieties of dimension $g$. For a vector $\Bigl[\begin{array}{c}
\epsilon\\
\delta\\
\end{array}\Bigr]=\left[\begin{array}{ccc}
\epsilon_1 \ldots \epsilon_g\\
\delta_1 \ldots \delta_g\\
\end{array}\right]\in \FF_2^{2g}$ one defines the \emph{Riemann theta function with characteristics} as the holomorphic function $\vartheta:\mathfrak{H}_g\times \mathbb C^g\rightarrow \mathbb C$ given by
$$\vartheta\left[\begin{array}{c}
\epsilon\\
\delta\\
\end{array}\right](\tau, z):=\sum_{m\in \mathbb Z^g} \mathrm{exp} \Bigl(\pi i \ ^t(m+\frac{\epsilon}{2})\tau (m+\frac{\epsilon}{2})+2\pi i \ ^t(m+\frac{\epsilon}{2})(z+\frac{\delta}{2})\Bigr).$$
For any period matrix $\tau \in \mathfrak{H}_g$, the pair
$$\Bigl[A_{\tau}:=\frac{\mathbb C^g}{\mathbb Z^g+\tau\cdot \mathbb Z^g}, \ \ \Theta_{\tau}:=\Bigl\{z\in A_{\tau}:\vartheta \left[\begin{array}{c}
0\\
0\\
\end{array}\right](\tau, z)=0\Bigr\}\Bigr]$$
defines a principally polarized abelian variety, that is, $[A_{\tau}, \Theta_{\tau}]\in \cA_g$.
There is an identification of symplectic vector spaces $$V:=\FF_2^{2g} \stackrel{\cong}\longrightarrow A_{\tau}[2],\ \mbox{ } \mbox{ given by } \ \mbox{ }\mbox{ }\left[\begin{array}{c}
\epsilon\\
\delta\\
\end{array}\right] \mapsto \frac{\tau\cdot \epsilon+\delta}{2}\in A_{\tau}[2].$$
This isomorphism being understood, points in $A_{\tau}[2]$ were classically called \emph{period characteristics}, see \cite{Kr} Section VII.2.
 The theta function $\vartheta\Bigl[\begin{array}{c}
\epsilon\\
\delta\\
\end{array}\Bigr](\tau, z)$ is the unique section of the translated line bundle $\OO_{A_{\tau}}\Bigl(t_{\frac{\tau\cdot \epsilon+\delta}{2}}^*(\Theta_{\tau})\Bigr)$. Krazer \cite{Kr} goes to great lengths to emphasize the difference between the \emph{period} and the \emph{theta} characteristics, even though at first sight, both sets of characteristics can be identified with the vector space $\FF_2^{2g}$. The reason for this is the transformation formula for theta functions under the action of the symplectic group, see \cite{BL} Theorem 8.6.1.  The theta constants with characteristics are modular forms of weight one half with respect to the subgroup $\Gamma_g(4, 8)$.  To define the action of the full symplectic group $\mathrm{Sp}_{2g}(\mathbb Z)$ on $\FF_2^{2g}$,  we consider the quadratic form associated to a characteristic
$\Delta:=\left[\begin{array}{c}
\epsilon\\
\delta\\
\end{array}\right]$, that is,
$$q_{\left[\begin{array}{c}
\epsilon\\
\delta\\
\end{array}\right]}\Bigl(x, y\Bigr)=x\cdot y+\epsilon\cdot x+ \delta\cdot y,$$
with $(x, y)=(x_1, \ldots, x_g, y_1, \ldots, y_g)\in \FF_2^{2g}$. We define an action of $\mathrm{Sp}_{2g}(\mathbb Z)$ on the set of characteristics which factors through the following action of $\mathrm{Sp}(\FF_2^{2g})$: If $M\in \mathrm{Sp}(\FF_2^{2g})$, set
$$q_{M\cdot \Delta}:=M \cdot q_{\Delta},$$
where we recall that we have already defined $(M\cdot q_{\Delta})(x, y):=q_{\Delta}(M^{-1}(x, y))$. 
It can be shown that this non-linear action of the symplectic group on the set of characteristics is compatible with the transformation rule of theta constants, see also \cite{CDvG}, \cite{BL}. In this context, once more, a theta characteristic $\left[\begin{array}{c}
\epsilon\\
\delta\\
\end{array}\right]$ appears as an element of $Q(V)$. This interpretation makes the link between the classical definition of theta characteristics found in \cite{Kr} and the more modern one encountered for instance in \cite{DO}, \cite{GH}. The algebra of period and theta characteristics was developed by Frobenius \cite{Fr1}, \cite{Fr2} in order to derive general identities between theta constants and describe the structure of the set of such identities.
\vskip 4pt

\subsection{Superstring scattering amplitudes and characteristic calculus.}
Recently, the action of $\mbox{Sp}(\FF_2^{2g})$ on the set of characteristics and the algebra of characteristic systems has been used by Grushevsky \cite{Gr}, Salvati Manni \cite{SM2}, Cacciatori, Dalla Piazza and van Geemen \cite{CDvG} and others, in order to find an explicit formula for the \emph{chiral superstring amplitudes}. This is an important foundational question in string theory and we refer to the cited papers for background
and further references.
Loosely speaking, D'Hoker and Phong  conjectured  that there exists a modular form $\Xi^{(g)}$ of weight $8$ with respect to the group $\Gamma_g(1, 2)\subset \Gamma_g=\mbox{Sp}_{2g}(\mathbb Z)$, satisfying the following two constraints:

\noindent (1) \emph{Factorization:} For each integer $1\leq k\leq g-1$, the following factorization formula
$$\Xi^{(g)}_{| \mathfrak{H}_k\times \mathfrak{H}_{g-k}}=\Xi^{(k)}\cdot \Xi^{(g-k)},$$
holds, when passing to the locus $\mathfrak{H}_k\times \mathfrak{H}_{g-k}\subset \mathfrak{H}_g$ of decomposable abelian varieties.

\noindent (2) \emph{Initial conditions:} For $g=1$, one must recover the standard chiral measure, that is,
$$\Xi^{(1)}=\theta\left[\begin{array}{c}
0\\
0\\
\end{array}\right]^8 \theta \left[\begin{array}{c}
0\\
1\\
\end{array}\right]^4
\theta \left[\begin{array}{c}
1\\
0\\
\end{array}\right]^4.
$$

A unique solution has been found in \cite{CDvG} when  $g=3$ and in \cite{SM2} when $g=4, 5$. This was placed in \cite{Gr} in a general framework that works in principle for arbitrary $g$.

For a set of theta characteristics $W\subset \FF_2^{2g}$, one defines the product of theta constants
$$P_W(\tau):=\prod_{\Delta\in W} \theta[\Delta](\tau, 0).$$
Note that $P_W$ vanishes if $W$ contains an odd characteristic. Then for all integers $0\leq i\leq g$, we set
$$P_i^{(g)}(\tau):=\sum_{W\subset \FF_2^{2g}, \ \mathrm{dim}(W)=i} P_W(\tau)^{2^{4-i}}.$$
It is pointed out in \cite{Gr} Proposition 13, that this function is a modular form of weight $8$ with respect to $\Gamma_g(1, 2)$. This result can be traced back to Frobenius.
Observe that in the definition of $P_i^{(g)}(\tau)$ the only non-zero summands correspond to totally syzygetic systems of characteristics, for else
such a system $W$ contains an odd characteristic and the corresponding term $P_W(\tau)$ is identically zero. Then it is showed that the expression
$$\Xi^{(g)}:=\frac{1}{2^g}\sum_{i=0}^g(-1)^i 2^{i\choose 2} P_i^{(g)}$$
satisfies the factorization rules when $g\leq 5$. Note that for higher $g$ the definition of $\Xi^{(g)}$ leads to a multivalued function due to the impossibility of choosing consistently the roots of unity for the various summands.
\begin{question}
Can one find the superstring amplitudes for higher $g$ by working directly with the space $\ss_g^+$ and constructing via algebro-geometric rather than theta function methods a system of effective divisors of slope (weight) $8$ satisfying the factorization formula?
\end{question}

\section{Cornalba's  moduli space of spin curves}

We now concern ourselves  with describing Cornalba's \cite{C} compactification $\ss_g$ of the moduli space $\cS_g$
of theta characteristics. We recall that $\cS_g$ is the parameter space of pairs $[C, \theta]$, where $C$ is a smooth curve of genus $g$ and $\theta\in \mbox{Th}(C)$. Following Atiyah \cite{At} such a pair is called a \emph{spin curve} of genus $g$. Mumford \cite{Mu} by algebraic means (and Atiyah \cite{At} with analytic methods) showed that the parity of a spin curve is locally constant in families: If $\phi:X\rightarrow S$ is a flat family of smooth curves of genus $g$ and $\mathcal{L}$ is a line bundle on $X$ together with a morphism $\beta:\mathcal{L}^{\otimes 2}\rightarrow \omega_{\phi}$ such that $\phi_s:\mathcal{L}_{X_s}^{\otimes 2}\rightarrow \omega_{X_s}$ is an isomorphism for all $s\in S$,  then the map
$ \mathrm{arf}(\phi):S\rightarrow \FF_2$ defined as  $\mbox{arf}(\phi)(s):=\mathrm{arf}(\mathcal{L}_{X_s})$
 is constant on connected components of $S$. One can speak of the parity of a family of spin curves and according to the value of
$\mbox{arf}(q_{\theta})$, the moduli space $\cS_g$ splits into two connected components $\cS_g^+$ and $\cS_g^-$. The forgetful map $\pi:\cS_g\rightarrow \cM_g$, viewed as a morphism of Deligne-Mumford stacks, is unramified. A compactified moduli space $\ss_g$ should be the coarse moduli space of a Deligne-Mumford stack, such that there is a \emph{finite} morphism $\pi:\ss_g\rightarrow \mm_g$ fitting into a commutative diagram:
$$\xymatrix{
  & \cS_g \ar[d]_{\pi} \ar@{^{(}->}@<-0.5ex>[r] & \ss_g \ar[d]_{\pi}& \\
  & \cM_g \ar@{^{(}->}@<-0.5ex>[r] & \mm_g & \\
}$$
As an algebraic variety, $\ss_g$ is the normalization of $\mm_g$ in the function field of $\cS_g$. It is the main result of \cite{C} that the points
of $\ss_g$ have a precise modular meaning in terms of line bundles on curves belonging to a slightly larger class than that of stable curves.

\begin{definition} A reduced, connected, nodal curve $X$ is called \emph{quasi-stable}, if for any component $E\subset X$ that is isomorphic to $\PP^1$, one has that (i) $k_E:=|E\cap \overline{(X-E)}|\geq 2$, and (ii) any two rational components $E, E'\subset X$ with $k_E=k_{E'}=2$, are disjoint.
\end{definition}
Smooth rational components $E\subset X$ for which $k_E=2$ are called \emph{exceptional}.
The class of quasi-stable curve is a very slight enlargement of the class of stable curves. To obtain a quasi-stable curve, one takes a stable curve $[C]\in \mm_g$ and a subset of nodes $N\subset \mbox{Sing}(C)$ which one "blows-up"; if $\nu:\tilde{C}\rightarrow C$ denotes the normalization map and $\nu^{-1}(n)=\{n^-, n^+\}$, then we define the nodal curve
$$X:=\tilde{C}\cup \Bigl(\bigcup_{n\in N} E_n\Bigr),$$
where $E_n=\PP^1$ for each node $n\in N$ and $E_n\cap (\overline{X-E_n})=\{n^+, n^-\}$. The \emph{stabilization map} $\mbox{st}:X\rightarrow C$ is a partial normalization and contracts all exceptional components $E_n$, that is, $\mbox{st}(E_n)=\{n\}$, for each $n\in N$.

\vskip 3pt
Why extend the class of stable curves, after all $\mm_g$ is already a projective variety?
One can define compactified moduli spaces of theta characteristics working with stable curves alone, if one is prepared to allow sheaves that are not locally free at the nodes of the curves. Allowing semi-stable curves, enables us to view a degeneration of a theta characteristic as a line bundle on a curve that is (possibly) more singular. These two philosophies of compactifying a parameter space of line bundles, namely restricting the class of curves but allowing singularities for the sheaves, vs. insisting on local freeness of the sheaves but enlarging the class of curves, can also be seen at work in the two (isomorphic) compactifications  of the universal degree $d$ Jacobian variety $P_{d, g}$ over $\cM_g$ constructed in \cite{Ca} and \cite{P}. We now describe all points of $\ss_g$:

\begin{definition}\label{spinstructures} A \emph{stable spin curve} of genus
$g$ consists of a triple $(X, \theta, \beta)$, where $X$ is a quasi-stable curve of arithmetic genus $g$, $\theta\in \mathrm{Pic}^{g-1}(X)$ is a line
bundle of total degree $g-1$ such that $\theta_{E}=\OO_E(1)$ for every exceptional component $E\subset X$, and $\beta:\theta^{\otimes
2}\rightarrow \omega_X$ is a sheaf homomorphism which is
not zero along each non-exceptional component of $X$.
\end{definition}
When $X$ is a smooth curve, then $\theta\in \mbox{Th}(X)$ is an ordinary theta characteristics and $\beta$ is an isomorphism. Note that in this definition, the morphism $\beta:\theta^{\otimes 2}\rightarrow \omega_X$ vanishes with order $2$ along each exceptional component $E\subset X$. Cornalba \cite{C} proved that stable spin curves form a projective moduli space $\ss_g$ endowed with a regular stabilization morphism $\pi:\ss_g\rightarrow \mm_g$,  set-theoretically given by $\pi([X, \eta, \beta]):=[C]$; here $C$ is obtained from $X$ by contracting all the exceptional components. The space $\ss_g$ has two connected components $\ss_g^-$ and $\ss_g^+$ depending on the parity $h^0(X, \theta) \mbox{ mod } 2$ of the spin structure.

\begin{definition}\label{thC}
For a stable curve $[C]\in \mm_g$, we denote by $\mbox{Th}(C):=\pi^{-1}([C])$ the zero-dimensional scheme of length $2^{2g}$ classifying stable spin structures on quasi-stable curves whose stable model is $C$.
\end{definition}

The scheme $\mbox{Th}(C)$ has an interesting combinatorial structure that involves the dual graph of $C$. Before describing it for an arbitrary curve $[C]\in \mm_g$, it is helpful to understand the boundary structure of $\ss_g$, which amounts to describe $\mbox{Th}(C)$ when $C$ is a $1$-nodal curve. We shall concentrate on the space $\ss_g^-$ and leave $\ss_g^+$ as an exercise (or refer to \cite{C}, \cite{F2} for details).

The boundary $\mm_g-\cM_g$ of the moduli space of curves decomposes into components $\Delta_0, \ldots, \Delta_{[\frac{g}{2}]}$. A general point of $\Delta_0$ corresponds to a $1$-nodal irreducible curve of arithmetic genus $g$, and for $1\leq i\leq [\frac{g}{2}]$ the general point of $\Delta_i$ is the class of the union of two components of genera $i$ and $g-i$ respectively, meeting transversely at a point.

\subsection{Spin curves of compact type.} We fix an integer $1\leq i\leq [\frac{g}{2}]$ and a general point $[C\cup_y D]\in \Delta_i$, where
$[C, y]\in \cM_{i, 1}$ and $[D, y]\in \cM_{g-i, 1}$ are smooth curves. We describe all stable spin curve $[X, \theta, \beta]\in \pi^{-1}([C\cup_y D])$.
For degree reasons, $X\neq C\cup_y D$, that is, one must insert an exceptional component $E$ at the node $y\in C\cap D$ and then
$$X:=C\cup_{y^+} E\cup_{y^-} D,$$ where
$C\cap E=\{y^+\}$ and $D\cap E=\{y^-\}$.
Moreover $$\theta=\bigl(\theta_C, \theta_D, \theta_E=\OO_E(1)\bigr)\in
\mbox{Pic}^{g-1}(X),$$ and since $\beta_{| E}=0$, it follows that  $\theta_C\in \mbox{Th}(C)$ and  $\theta_D\in \mbox{Th}(D)$, that is, a theta characteristic on a curve of compact type is simply a collection of theta characteristics on each of its (necessarily smooth) components.
The condition that $h^0(X, \theta)$ be odd
implies that $\theta_C$ and $\theta_D$ have
opposite parities. Accordingly, the pull-back divisor $\pi^*(\Delta_i)$ splits in two components depending on the choice of the respective Arf invariants:  We denote by $A_i\subset \ss_g^-$ the closure of
the locus corresponding points for which $\theta_C=\theta_C^-$ and $\theta_D=\theta_D^+$, that is, $\mbox{arf}(\theta_C)=1$ and $\mbox{arf}(\theta_D)=0$.
We denote by $B_i\subset \ss_g^-$ the
closure of the locus of spin curves for which $\mbox{arf}(\theta_C)=0$ and $\mbox{arf}(\theta_D)=1$.
At the level of divisors, the following relation holds $$\pi^*(\Delta_i)=A_i+B_i.$$
Moreover one has that 
$$\mbox{deg}(A_i/\Delta_i)=2^{g-2}(2^i-1)(2^{g-i}+1)\
\mbox{ and }\ \mbox{deg}(B_i/\Delta_i)=2^{g-2}(2^i+1)(2^{g-i}-1).$$

\subsection{Spin curves with an irreducible stable model.} We fix a general $2$-pointed smooth curve  $[C, x, y]\in \cM_{g-1, 2}$ and identify the points $x$ and $y$. The resulting stable curve $\nu:C\rightarrow C_{xy}$, where $C_{xy}:=C/x\sim y$,  corresponds to a general point of the boundary divisor $\Delta_0$. Unlike in the case of curves of compact type, two possibilities do occur,
depending on whether $X$ possesses an exceptional component or not.

If $X=C_{xy}$, then the locally free sheaf $\theta$ is a root of the dualizing sheaf $\omega_{C_{xy}}$. Setting  $\theta_C:=\nu^*(\theta)\in \mbox{Pic}^{g-1}(C)$, from the condition $$H^0(C, \omega_C(x+y)\otimes \theta_C^{\otimes (-2)})\neq 0,$$ by counting degrees, we obtain that $\theta_C^{\otimes 2}=K_C(x+y)$.
For each choice of $\theta_C\in \mathrm{Pic}^{g-1}(C)$ as above, there
is precisely one choice of gluing the fibres $\theta_C(x)$ and
$\theta_C(y)$ in a way that  if $\theta$ denotes the line bundle on $C_{xy}$ corresponding to this gluing, then $h^0(X, \theta)$ is odd.
Let $A_0$ denote the closure in $\ss_g^-$ of the locus of such points. Then
$\mbox{deg}(A_0/\Delta_0)=2^{2g-2}$. Since we expect the fibre $\mbox{Th}(C_{xy})$ to consist of $2^{2g}$ points (counting also multiplicities), we see that one cannot recover all stable spin curves having $C_{xy}$ as their stable model by considering square roots of the dualizing sheaf of $C_{xy}$ alone. The remaining spin curves correspond to sheaves on $C_{xy}$ which are not locally free at the node, or equivalently, to spin structures on a strictly quasi-stable curve.
\vskip 3pt

Assume now that $X=C\cup_{\{x, y\}} E$, where $E$ is an exceptional component. Since $\beta_{| E}=0$
it follows that $\beta_{| C}\in H^0(C, \omega_{X| C}\otimes \theta_C^{\otimes (-2)})$ must vanish at both
$x$ and $y$ and then for degree reasons $\theta_C\in \mbox{Th}(C)$. For parity reasons, $\mbox{arf}(\theta_C)=1$.
We denote by $B_0$ the closure in $\ss_g^-$ of the locus of such points.
A local analysis carried out in \cite{C} shows that $\pi$ is simply ramified over $B_0$. Since $\pi:\ss_g^-\rightarrow \mm_g$ is not ramified along any other divisors of $\ss_g^-$,  one deduces that $B_0$ is the ramification divisor of the forgetful map and  the following relation
holds:
$$
\pi^*(\Delta_0)=A_0+2B_0.
$$
A general point of $B_0$ is determined by specifying an odd theta characteristic on $C$, thus  $\mbox{deg}(B_0/\Delta_0)=2^{g-2}(2^{g-1}-1)$. By direct calculation one checks that
$$\mbox{deg}(A_0/\Delta_0)+2\mbox{deg}(B_0/\Delta_0)=2^{g-1}(2^g-1),$$
which confirms that $\pi$ is simply ramified along $B_0$.

\vskip 4pt

After this preparation, we are now ready to tackle the case of an arbitrary stable curve $C$. We denote by $\nu:\tilde{C}\rightarrow C$ the normalization map and by $\Gamma_C$ the \emph{dual graph} whose vertices are in correspondence with components of $C$, whereas an edge of $\Gamma_C$ corresponds to a node which lies at the intersection of two components (note that self-intersections are allowed). A set of nodes $\Delta\subset \mbox{Sing}(C)$ is said to be \emph{even}, if for any component $Y\subset C$, the degree $|\nu^{-1}(Y\cap \Delta)|$ is an even number. For instance, if $C:=C_1\cup C_2$ is a union of two smooth curves meeting transversally, a set of nodes $\Delta\subset C_1\cap C_2$ is even if and only if $|\Delta|\equiv 0\mbox{ mod } 2$.

\begin{remark} Assume $[X, \theta, \beta]\in \ss_g$ and let $\mathrm{st}:X\rightarrow C$ be the stabilization morphism. If $N\subset \mbox{Sing}(C)$ is the set of exceptional nodes of $C$, that is, nodes $n\in N$ having the property that
$\mathrm{st}^{-1}(n)=E_n=\PP^1$, and we write that $\mbox{Sing}(C)=N\cup \Delta$, then the set $\Delta$ of non-exceptional nodes
of $C$ is even, see both \cite{C} and \cite{CC}.
\end{remark}

One has the following description of the scheme $\mbox{Th}(C)$, cf. \cite{CC} Proposition 5:
\begin{proposition}\label{components}
Let $[C]\in \mm_g$ and $b:=b_1(\Gamma_C)$ be the Betti number of the dual graph. Then the number of components of the zero-dimensional scheme $\mathrm{Th}(C)$
is equal to
$$2^{2g-2b}\cdot \Bigl(\sum_{\Delta\subset \mathrm{Sing}(C), \Delta\ \mathrm{ even}} 2^{b_1(\Delta)}\Bigr).$$
A component corresponding to an even set $\Delta\subset \mathrm{Sing}(C)$ appears with multiplicity $2^{b-b_1(\Delta)}$.
\end{proposition}
One can easily verify that the length of the scheme $\mbox{Th}(C)$ is indeed $2^{2g}$. From Proposition \ref{components} it follows for instance that $\mbox{Th}(C)$ is a reduced scheme if and only if $C$ is of compact type. In the case we studied above, when $C$ is irreducible with a single node, Proposition \ref{components} gives that $\mbox{Th}(C)=\mbox{Th}(C)_-\cup \mbox{Th}(C)_+$ has $3\cdot 2^{2g-2}$ irreducible components. Precisely, $|\mbox{Th}(C)_-|=2^{2g-2}+|\mbox{Th}(\tilde{C})_-|$ and $|\mbox{Th}(C)_+|=2^{2g-2}+|\mbox{Th}(\tilde{C})_+|$.
\begin{remark}
The structure of the schemes $\mbox{Th}(C)$ for special singular curves $C$ has been used in \cite{Lud} to describe the singularities of $\ss_g$ and in \cite{CS} to prove that a general curve $[C]\in \cM_g$ is uniquely determined by the set of contact hyperplanes $$\bigl\{\langle D\rangle \in (\PP^{g-1})^{\vee}: D\in |\theta|, \ \theta\in \mathrm{Th}(C)_-\bigr\}.$$
In spite of these important applications, a systematic study of the finite geometry of the set $\mbox{Th}(C)$ when $C$ is singular (e.g. a theory of fundamental and G\"opel systems, syzygetic tetrads) has not yet been carried out and is of course quite interesting.
\end{remark}

\vskip 3pt
\subsection{The canonical class of the spin moduli space.}
It is customary to denote the divisor classes in the Picard group of the moduli stack by
\begin{center}
\fbox{$\alpha_i:=[A_i], \ \ \beta_i:=[B_i]\in \mbox{Pic}(\ss_g^-), \ \mbox{  } i=0, \ldots, [\frac{g}{2}].$}
\end{center}
A result of Putman's \cite{Pu} shows that for $g\geq 5$, the divisor classes
$\lambda$, $\alpha_0$, $\beta_0, \ldots, \alpha_{[\frac{g}{2}]}$, $\beta_{[\frac{g}{2}]}$ freely generate the rational Picard group $\mbox{Pic}(\ss_g^-)$. A similar result holds for $\ss_g^+$.

The space $\ss_g^-$ is a normal variety with finite quotient singularities; an \'etale neighbourhood of an arbitrary point $[X, \eta, \beta]\in \ss_g$ is of the form $$H^0(C, \omega_C\otimes \Omega_C)^{\vee}/\mbox{Aut}(X, \eta, \beta)=\mathbb C^{3g-3}/\mbox{Aut}(X, \eta, \beta),$$
where $H^0(C, \omega_C\otimes \Omega_C)$ can be identified via deformation theory with the cotangent space to the moduli stack $\mm_g$ at the point $[C]:=\pi([X, \eta, \beta])$.  For a (predictable) definition of an automorphism of a triple $(X, \eta, \beta)$, we refer to \cite{C} Section 1.
\vskip 3pt

Using Kodaira-Spencer deformation theory, one describes the cotangent bundle of the stack $\mm_g$ as the push-forward of a rank $1$ sheaf on the universal curve over $\mm_g$. A famous and at the time, very innovative use in \cite{HM} of the Grothendieck-Riemann-Roch theorem for the universal curve, yields the formula $$K_{\mm_g}\equiv 13\lambda-2\delta_0-3\delta_1-2\sum_{i=2}^{[\frac{g}{2}]} \delta_i\in \mbox{Pic}(\mm_g).$$
From the Riemann-Hurwitz theorem applied to the finite branched cover $\pi:\ss_g^-\rightarrow \mm_g$, we find the formula for the canonical class of the spin moduli stack:

\begin{center}
\fbox{
$K_{\ss_g^-}\equiv 13\lambda-2\delta_0-3\beta_0-3(\alpha_1+\beta_1)-2\sum_{i=2}^{[\frac{g}{2}]} (\alpha_i+\beta_i)\in \mbox{Pic}(\ss_g^-)$.}
\end{center}

An identical formula holds for $\ss_g^+$. Unfortunately both spaces $\ss_g^+$ and $\ss_g^-$ have non-canonical singularities, in particular there exist \emph{local} obstructions to extending pluri-canonical forms defined on the smooth part of $\ss_g$ to a resolution of singularities. However, an important result of Ludwig \cite{Lud} shows that this obstructions are not of global nature. The following result holds for both $\ss_g^+$ and $\ss_g^-$:
\begin{theorem}(Ludwig) For $g\geq 4$ fix a resolution of singularities $\epsilon:\widetilde{\cS}_g\rightarrow \ss_g$. Then for any integer $\ell\geq 0$ there exists an isomorphism of vector spaces
$$\epsilon^*: H^0(\ss_g, K_{\ss_{g, \mathrm{reg}}}^{\otimes \ell})\stackrel{\cong}\longrightarrow H^0(\widetilde{\cS}_g, K_{\widetilde{\cS}_g}^{\otimes \ell}).$$
\end{theorem}

Therefore, in order to conclude that  $\ss_g$ is of general type, it suffices to show that the canonical class $K_{\ss_g}$ lies in the interior of the effective cone $\mbox{Eff}(\ss_g)$ of divisors, or equivalently, that it can be expressed as a positive linear combination of an ample and an effective class on $\ss_g$. This becomes a question on slopes of the effective cone of $\ss_g$, which can be solved in a spirit similar to
\cite{HM} and \cite{EH}, where it has been proved with similar methods that $\mm_g$ is of general type for $g\geq 24$. The standard references for effective divisors on moduli spaces of stable curves are \cite{HM}, \cite{EH}, \cite{Log} and \cite{F4}.

\section{Effective divisors on $\ss_g$}

In the papers \cite{F2}, \cite{F3} and \cite{FV1} we have initiated a study of effective divisors on $\ss_g$. The class of the locus $\thet$ of vanishing theta-nulls on $\ss_g^+$ is computed in \cite{F2}; in the paper \cite{FV1} we study the space $\ss_g^-$ with the help of the divisor of spin curves with an everywhere tangent hyperplane in the canonical embedding having higher order contact than expected. Finally in \cite{F3} we define effective divisors of Brill-Noether type in more general setting on both spaces $\ss_g^-$ and $\ss_g^+$. We survey these constructions, while referring to the respective papers for technical details.

We begin with the moduli space $\ss_g^-$: An odd theta characteristic $\theta\in \mbox{Th}(C)$ with $h^0(C, \theta)=1$ determines a unique effective divisor $D\in C_{g-1}$ such that $\theta=\OO_C(D)$. We write in this case that $D=\mbox{supp}(\theta)$.
The assignment $(C, \theta)\mapsto \bigl(C, \mbox{supp}(\theta)\bigr)$ can be viewed as a rational map between moduli spaces
$$\ss_g^-\dashrightarrow \cc_{g, g-1}:=\mm_{g, g-1}/\mathfrak{S}_{g-1},$$
and it is natural to use this map and the well-understood divisor theory \cite{Log} on the universal symmetric product $\cc_{g, g-1}$, in order to obtain promising effective divisors on $\ss_g^-$. In particular, the boundary divisor $\Delta_{0:2}$ on $\cc_{g, g-1}$ with general point being a  pair $(C, D)$ where $D\in C_{g-1}$ is a divisor with non-reduced support, is known to be extremal. Its pull-back to $\ss_g^-$  parametrizes (limits of) odd spin curves $[C, \theta]\in \cS_g^-$ such that there exists a point $x\in C$ with $H^0(C, \theta(-2x))\neq 0$. The calculation of the class of the closure of this locus is one of the main results of \cite{FV1}:
\begin{theorem}\label{degen}
We fix $g\geq 3$. The locus consisting of odd spin curves
$$\cZ_g:=\bigl\{[C, \theta]\in \cS_g^-:\theta=\OO_C\bigl(2x_1+\sum_{i=2}^{g-2} x_i\bigr),\ \mbox{ with } x_i\in C \mbox{ for } i=1, \ldots,  g-2 \bigr\}$$
is a divisor on $\cS_g^-$. The class of its compactification inside $\ss_g^-$ equals
$$[\overline{\cZ}_g]= (g+8)\lambda-\frac{g+2}{4}\alpha_0-2\beta_0-\sum_{i=1}^{[\frac{g}{2}]} 2(g-i)\ \alpha_i-\sum_{i=1}^{[\frac{g}{2}]} 2i\ \beta_i\in \mathrm{Pic}(\ss_g^-).$$
\end{theorem}
For low genus, $\cZ_g$ specializes to well-known geometric loci. For instance $\cZ_3$ is the divisor of hyperflexes on plane quartics, classifying
pairs $[C, \OO_C(2p)]\in \cS_3^-$, where $p\in C$ is such that $h^0(C, \OO_C(4p))=3$. Then $K_C=\OO_C(4p)$ and $p\in C$ is a hyperflex point.

The divisor $\zz_g$ together with pull-backs of effective divisors on $\mm_g$ can be used to determine the range in which $\ss_g^-$ is of general type. This application also comes from \cite{FV1}:
\begin{theorem}\label{sg-}
The moduli space $\ss_g^-$ is a variety of general type for $g\geq 12$.
\end{theorem}

The passing from Theorem \ref{degen} to Theorem \ref{sg-} amounts to simple linear algebra. By comparing the class $[\zz_g]$ against that of the canonical divisor, we note that $$K_{\ss_g^-}\notin \mathbb Q_{\geq 0}\Bigl\langle [\overline{\cZ}_g], \lambda, \alpha_i, \beta_i, \ i=0, \ldots, \bigl[\frac{g}{2}\bigr]\Bigr\rangle,$$ for the coefficient of $\beta_0$ in the expression of $[\zz_g]$ is too small. However one can combine $[\zz_g]$ with effective classes coming from $\mm_g$, and there is an ample supply of such, see \cite{HM}, \cite{EH}, \cite{F4}. To avoid technicalities, let us assume that $g+1$ is composite and consider the Brill-Noether divisor $\cM_{g, d}^r$  of curves $[C]\in \cM_g$ with a $\mathfrak g^r_d$, where the Brill-Noether number $\rho(g, r, d):=g-(r+1)(g-d+r)=-1$. The class of the closure $\mm_{g, d}^r$ of $\cM_{g, d}^r$ in $\mm_g$ has been computed in \cite{EH} (and in \cite{HM} for $r=1$) and plays a crucial role in the proof by Harris, Mumford, Eisenbud that the moduli space $\mm_g$ is of general type for $g\geq 24$. There exists an explicit constant $c_{g, d, r}>0$ such that the following relation holds \cite{EH},
$$[\mm_{g, d}^r]= c_{g, d, r}\Bigl((g+3)\lambda-\frac{g+1}{6}\delta_0-\sum_{i=1}^{[\frac{g}{2}]} i(g-i)\delta_i\Bigr)\in \mathrm{Pic}(\mm_g).$$
By interpolation, one find a  constant $c'_{g, d, r}>0$ such that the effective linear combination
$$\frac{2}{g-2}[\zz_g]+c'_{g, d, r}[\pi^*(\mm_{g, d}^r)]=\frac{11g+37}{g+1}\lambda-2\alpha_0-3\beta_0-\sum_{i=1}^{[g/2]} (a_i\cdot \alpha_i+ b_i),$$
where $a_i, b_i\geq 2$ for $i\neq 1$ and $a_1, b_1>3$ are explicitly known rational constants. By comparison, whenever the inequality
$$\frac{11g+37}{g+1}<13\Leftrightarrow g>12$$
is satisfied (and $g+1$ is composite), the class $K_{\ss_g^-}$ is big, that is, $\ss_g^-$ is of general type. The case $g=12$ is rather difficult and we refer to the last section of \cite{FV1}.
\vskip 3pt

\subsection{The locus of curves with a vanishing theta-null.}
On $\ss_g^+$ we consider the locus of even spin curves with a vanishing theta characteristic. The following comes from \cite{F2}:
\begin{theorem}\label{thetanull}
The closure in $\ss_g^{+}$ of the divisor
$$\Theta_{\mathrm{null}}:=\bigl\{[C, \eta]\in \cS_g^{+}: H^0(C, \eta)\neq
0\bigr\}$$ of curves with an effective even theta characteristics has class equal to
$$[\overline{\Theta}_{\mathrm{null}}]=
\frac{1}{4}\lambda-\frac{1}{16}\alpha_0-\frac{1}{2}\sum_{i=1}^{[\frac{g}{2}]}
 \beta_i\in \   \mathrm{Pic}(\ss_g^{+}).$$
\end{theorem}
\noindent In the paper \cite{FV2} it is shown that the class $[\thet]\in \mbox{Eff}(\ss_g^+)$ is extremal when $g\leq 9$. It is an open question whether this is the case for arbitrary $g$, certainly there is no known counterexample to this possibility. Combining $[\thet]$ with pull-backs of effective classes from $\mm_g$ like in the previous case, we find a  constant $c^{''}_{g, d, r}>0$  such that
$$8[\overline{\Theta}_{\mathrm{null}}]+c^{''}_{g, d, r}[\pi^*(\mm_{g, d}^r)]=\frac{11g+29}{g+1}\lambda-2\alpha_0-3\beta_0-\sum_{i=1}^{[\frac{g}{2}]}
(a_i'\cdot \alpha_i+b_i'\cdot \beta_i)\in \mathrm{Eff}(\ss_g^+),$$
where the appearing coefficients satisfy the inequalities $a_i', b_i'\geq 2$ for $i\geq 2$ and $a_1', b_1'>3$. Restricting ourselves again to the case when $g+1$ is composite (while referring to \cite{F2} for the remaining cases), we obtain that whenever
$$\frac{11g+29}{g+1}<13\Leftrightarrow g>8,$$ the space $\ss_g^+$ has maximal Kodaira dimension. We summarize these facts as follows:
\begin{theorem}\label{sg+}
The moduli space $\ss_g^+$ is a variety of general type for $g>8$.
\end{theorem}

\section{Unirational parametrizations of $\cS_g$ in small genus}
Next we present ways of proving the unirationality of $\cS_g$ in small genus. The basic references are the papers \cite{FV1} and \cite{FV2}.
We recall that a normal $\mathbb Q$-factorial projective variety $X$ is said to be \emph{uniruled} if through a very general point $x\in X$ there passes a rational curve $R\subset X$. Uniruled varieties have negative Kodaira dimension. Conversely, if the canonical class $K_X$ is not a limit of effective divisor classes (which implies that the Kodaira dimension of $X$ is negative), then $X$ is uniruled, see \cite{BDPP}.

The classification by Kodaira dimension of both $\ss_g^-$ and $\ss_g^+$  is governed by $K3$ surfaces, in the sense that
$\cS_g$ is uniruled precisely when a general spin curve $[C, \theta]\in \cS_g$ can be represented as a section of a special $K3$ surface $S$.  By varying $C$ in a pencil on $S$, we induce a rational curve in the moduli space $\ss_g$ passing through a general point. The $K3$ surface must have special properties that will allow us to assign a theta characteristic to each curve in the pencil. In the case of even spin curves, the $K3$ surface in question must be of Nikulin type. We refer to \cite{vGS} for an introduction to Nikulin surfaces.
\begin{definition}
A polarized \emph{Nikulin surface} of genus $g\geq 2$ consists of a triple  $(S, e, \OO_S(C))$, where $S$ is a smooth $K3$ surface, $e\in \mbox{Pic}(S)$ is a non-trivial line bundle with the property that $e^{\otimes 2}=\OO_S(N_1+\cdots+N_8)$, where $N_1, \ldots, N_8$ are pairwise disjoint $(-2)$ curves on $S$, and $C\subset S$ is a numerically effective curve class such that $C^2=2g-2$ and $C\cdot N_i=0$, for $i=1, \ldots, 8$.
\end{definition}
Since the line bundle $\OO_S(N_1+\ldots+N_8)$ is divisible by two, there exists a double cover $f:\widetilde{S}\rightarrow S$ branched exactly along the smooth rational curves $N_1, \ldots, N_8$. The curve $C\subset S$ does not meet the branch locus of $f$, hence the restriction $f_{| f^{-1}(C)}: f^{-1}(C)\rightarrow C$ is an unramified double covering which induces a non-trivial half-period on $C$. Each curve in the linear system $|\OO_S(C)|$ acquires a half-period in its Jacobian in this way. The following result is quoted from \cite{FV2}:
\begin{theorem}\label{eventheta}
The even spin moduli space $\ss^+_g$ is uniruled for $g \leq 7$.
\end{theorem}

\begin{proof}
Let us choose a general spin curve $[C, \theta]\in \cS_g^+$ and a non-trivial point of order two $\eta\in \mathrm{Pic}^0(C)[2]$, such that $h^0(C, \theta\otimes \eta)\geq 1$. Because of the generality assumption it follows that $h^0(C, \theta\otimes \eta)=1$ and the support of $\theta\otimes \eta$ consists of $g-1$ distinct points $p_1, \ldots, p_{g-1}$. To simplify matters assume that $g\neq 6$. Then it is proved in \cite{FV2} that the general Prym curve $[C, \eta]\in \cR_g$ is a section of a genus $g$ polarized Nikulin surface $(S, e)$, that is, $C\subset S$ and $\eta=e\otimes \OO_C$. We consider the map induced by the linear system
 $$
 \varphi_{|\OO_S(C)|}: S\rightarrow \PP^g.
 $$
The points $\varphi(p_1), \ldots, \varphi(p_{g-1})$ span a codimension $2$ linear subspace. Let
$\PP\subset |\OO_S(C)|$
be the pencil of curves on $S$ induced by the hyperplanes in $\PP^g$ through $\varphi(p_1), \ldots, \varphi(p_{g-1})$. Each curve $C' \in \PP$ contains the divisor $p_1+\cdots+p_{g-1}$ as an \emph{odd} theta characteristic. The line bundle $\OO_{C'}(p_1+\cdots+p_{g-1})\otimes e_{C'}\in \mathrm{Pic}^{g-1}(C')$ is an \emph{even} theta characteristic on each curve $C'$, because as already discussed, the Arf invariant remains constant in a family of spin curves. This procedure induces a rational curve in moduli
$$
m: \PP \to \ss^+_g, \ \ \mbox{  }  m(C'):=[C', e\otimes \OO_{C'}(p_1+\cdots+p_{g-1})],
$$
which passes through the general point $[C, \theta]\in \ss_g^+$ and finishes the proof.
\end{proof}

Observe that in this proof, if instead of being a Nikulin surface, $S$ is an arbitrary $K3$ surface containing $C$, the same reasoning can be used to construct a rational curve in $\ss_g^-$ that passes through a general point, provided the curve $C$ we started with,  has general moduli.  A general curve of genus $g$ lies on a $K3$ surface if an only if $g\leq 9$ or $g=11$, see \cite{M1}. The case $g=10$ can be handled via a slightly different idea, see \cite{FV1} Theorem 3.10. Thus one also has the following result:
\begin{theorem}
The odd spin moduli space $\ss_g^-$ is uniruled for $g\leq 11$.
\end{theorem}

\subsection{Odd theta characteristics and Mukai models of $\cS_g^-$.} We explain the strategy pursued in \cite{FV1} to construct
alternative models of moduli spaces of odd spin curves which can then be used to establish unirationality of the moduli space:

\begin{theorem}\label{sgunir}
$\ss_g^-$ is unirational for $g\leq 8$.
\end{theorem}

The main idea is to construct a dominant map over $\ss_g^-$ from the total space of a projective bundle over a space parametrizing spin curves on nodal curves of smaller genus. We begin by recalling that Mukai, in a series of well-known papers \cite{M1}, \cite{M2}, \cite{M3}, \cite{M4}, has found ways of representing a general canonical curve of genus $g\leq 9$ as a linear section of a certain $n_g$-dimensional rational homogeneous variety $V_g\subset \PP^{n_g+g-2}$, which we shall call the \emph{Mukai variety} of genus $g$. One has the following list:\bigskip  \par \noindent \it
- $V_9$:  the Pl\"ucker embedding of the symplectic Grassmannian $SG(3,6)\subset \PP^{13}$, \par \noindent
- $V_8$:  the Pl\"ucker embedding of the Grassmannian $G(2,6)\subset \PP^{14}$, \par \noindent
- $V_7$:  the Pl\"ucker embedding of the orthogonal Grassmannian $OG(5,10)\subset \PP^{15}$. \bigskip  \par \noindent   \rm
Inside the Hilbert scheme of curvilinear sections of $V_g$, we denote by $\cU_g$ the open subset
classifying nodal sections $C\subset V_g$ by a linear space of dimension $g-1$.
The automorphism group $\mbox{Aut}(V_g)$ acts on $\cU_g$ and we call the GIT quotient
$$\mathfrak{M}_g:=\cU_g\dblq \mbox{Aut}(V_g)$$
the \emph{Mukai model} of the moduli space of curves of genus $g$. Note that with our definition, the variety $\mathfrak{M}_g$ is only quasi-projective and the Picard number of $\mathfrak{M}_g$ is equal to $1$. The moduli map $\cU_g\rightarrow \mm_g$ being $\mbox{Aut}(V_g)$-invariant,
it induces a regular map $\phi_g:\mathfrak{M}_g\rightarrow \mm_g$. We can paraphrase Mukai's results as stating that the map
$\phi_g: \mathfrak{M}_g \rightarrow \mm_g$ is a birational isomorphism, or equivalently, the general $1$-dimensional linear section of $V_g$ is a curve with general moduli. The map $\phi_g$ deserves more study and in principle it can be used to answer various questions concerning the cohomology of $\mm_g$ or the minimal model program of the moduli space of curves (see \cite{Fed} for a case in point when $g=4$).
The following concept is key to our parametrization of $\ss_g^-$ using Mukai models.

\begin{definition} Let $\mathfrak Z_{g-1}$ be the space of \emph{clusters}, that is, $0$-dimensional schemes $Z\subset V_g$ of length $2g-2$ with the following properties:
\begin{itemize}
\item[(1)] $Z$ is a hyperplane section of a smooth curve section $[C] \in \cU_g$,
\item[(2)] $Z$ has multiplicity two at each point of its support,
\item[(3)] $\mathrm{supp}(Z)$ consists of $g-1$ linearly independent points.
\end{itemize} \par \noindent
\end{definition} \par \noindent
 A general point of $\mathfrak{Z}_{g-1}$ corresponds to a $0$-cycle $p_1+\cdots+p_{g-1}\in \mbox{Sym}^{g-1}(V_g)$ satisfying $$\mbox{dim } \langle p_1, \ldots, p_{g-1}\rangle\cap \mathbb T_{p_i}(V_g)\geq 1, \mbox{ for } \ i=1, \ldots, g-1.$$ Furthermore, $\mathfrak{Z}_{g-1}$ is birational  to the subvariety of the Grassmannian variety $G(g-1, n_g+g-1)$ parametrizing $(g-2)$-dimensional planes $\Lambda\subset \PP^{n_g+g-2}$ such that $\Lambda \cdot V_g=2p_1+\cdots+2p_{g-1}$, where $p_1, \ldots, p_{g-1}\in V_g$. Then we consider the incidence correspondence:

\begin{center}
\fbox{
$\cU^-_g := \bigl\{ (C,Z) \in \cU_g \times \mathfrak{Z}_{g-1}: Z \subset C \bigr\}.$
}
\end{center}

The first projection map $\cU^-_g \to \cU_g$
is finite of degree $2^{g-1}(2^g-1)$; its fibre at a general point $[C]\in \cU_g$ corresponding to a smooth curve classifies odd theta
characteristics of $C$.  The spin moduli map $\cU^-_g \dashrightarrow \overline {\mathcal S}_g^-$ induces a birational isomorphism
$$\phi_g^-:\cU_g^-\dblq \mathrm{Aut}(V_g)\rightarrow \ss_g^-.$$

Let us fix now an integer $0 \leq \delta\leq g-1$. We define the locally closed set of pairs consisting of clusters and $\delta$-nodal curvilinear sections of $V_g$, that is,
$$\cU_{g, \delta}^- := \bigl\{(\Gamma, Z)\in \cU_g^- : \mathrm{sing}(\Gamma)\subset \mbox{supp}(Z)\ \ \mathrm{ and } \ \ |\mathrm{sing}(\Gamma)| = \delta \bigr\}.
$$
The quotient of $\cU_{g, \delta}^-$ under the action of the automorphism group of $V_g$ is birational to the locus $B_{g, \delta}^-\subset \ss_g^-$ with general point given by an odd spin structure on a curve whose stable model is an irreducible $\delta$-nodal curve where each of the nodes is "blown-up" and an exceptional component is inserted. Let us fix a general point $(\Gamma, Z)\in \cU_{g, \delta}^-$ and suppose that $\mbox{Sing}(\Gamma)=\{p_1, \ldots, p_{\delta}\}$ and denote the $p_{\delta+1}, \ldots, p_{g-1}\in \Gamma_{\mathrm{reg}}$ the remaining points in the support of $Z$. If $\nu:N\rightarrow \Gamma$ is the normalization map, then we observe that $\OO_N(p_{\delta+1}+\cdots+p_{g-1})\in \mathrm{Th}(N)_-$, which gives rise to a point in the locus $B_{g, \delta}^-$.

\vskip 3pt
The important point now is that over $\cU_{g, \delta}^-$ one can consider an incidence correspondence that takes into account not only a $\delta$-nodal curve together with a cluster, but also all linear sections of $V_g$ that admit the same cluster. Precisely:
\begin{center}
\fbox{$\mathcal{P}_{g, \delta} := \bigl\{ \bigl(C, (\Gamma, Z)\bigr) \in \cU_g \times \cU^-_{g, \delta} \ : \ Z \subset C \bigr\}.$}
\end{center}

\vskip 3pt
The variety $\mathcal{P}_{g, \delta}$ comes equipped with projection maps
$$
\begin{CD}
{\cU^-_g} @<{\alpha}<< {\mathcal{P}_{g, \delta}} @>{\beta}>> {\cU^-_{g, \delta}}. \\
\end{CD}
$$
It is shown in \cite{FV1} that $\mathcal{P}_{g, \delta}$ is birational to a projective bundle over $\cU_{g, \delta}^-$ and furthermore,  the quotient $\PP_{g, \delta}^-:=\mathcal{P}_{g, \delta}\dblq \mbox{Aut}(V_g)$ is a projective bundle over $B_{g, \delta}^-$. Moreover, it is proved that the projection map
$\alpha$ is dominant if and only if
$$\delta\leq n_g-1.
$$
To summarize these considerations, we have reduced the unirationality of $\ss_g^-$ to two conditions. One is numerical and depends solely on the Mukai variety $V_g$, the other has to do with the geometry of the spin moduli space $B_{g, \delta}^-$ of nodal curves of smaller geometric genus:
\begin{theorem}
For $g\leq 9$, if $n_g$ denotes the dimension of the corresponding Mukai variety $V_g$, the moduli space $\ss_g^-$ is unirational provided there exists an integer $1\leq \delta\leq g-1$ such that
\begin{enumerate}
\item $\delta\leq n_g-1$, \
\item $B_{g, \delta}^-$  is unirational.
\end{enumerate}
\end{theorem}
It turns out that the locus $B_{g, g-1}^-$ is unirational for $g\leq 10$ (see \cite{FV1} Theorem 4.16). However condition (i) is only satisfied when $g\leq 8$, and this is the range for which Theorem \ref{sgunir} is known at the moment.

\section{Geometric aspects of moduli spaces of theta characteristics}

In this section we discuss a few major themes related to aspects of the geometry of $\cS_g$ other than birational classification.

\subsection{The Brill-Noether stratification of $\cS_g$.}
One can stratify theta characteristics by their number of global sections. For an integer $r\geq -1$ let us denote by
$$\cS_g^r:=\bigl\{[C, \theta]\in \cS_g: h^0(C, \theta)\ge r+1,\ \ h^0(C, \theta) \equiv r+1 \mbox{ mod }2\bigr\}.$$
The variety $\cS_g^r$ has a Lagrangian determinantal structure discussed in \cite{H} Theorem 1.10, from which it follows that each component
of $\cS_g^r$ has codimension at most ${r+1\choose 2}$ inside $\cS_g$. This bound also follows from  Nagaraj's \cite{Na} interpretation of the tangent space to the stack $\cS_g^r$ which we briefly explain. Fix a point $[C, \theta]\in \cS_g^r$ and form the Gaussian map
$$\psi_{\theta}:\wedge^2 H^0(C, \theta)\rightarrow H^0(C, K_C^{\otimes 2}), \mbox{ } \ s\wedge t\mapsto s\cdot dt-t\cdot ds.$$
More intrinsically, the projectivization of the map $\psi_{\theta}$ assigns to a pencil $\langle s, t\rangle \subset |\theta|$ its ramification divisor. Recalling the identification provided by Kodaira-Spencer theory
$$T_{[C, \theta]}(\cS_g)=H^0(C, K_C^{\otimes 2})^{\vee},$$
it is shown in \cite{Na} that the following isomorphism holds:
\begin{center}
\fbox{$T_{[C, \theta]}(\cS_g^r)=\bigl\{\varphi\in H^0(C, K_C^{\otimes 2})^{\vee}: \varphi_{| \mathrm{Im}\ \psi_{\theta}}=0\bigr\}.$}
\end{center}
This description is consistent with the bound $\mbox{codim}(\cS_g^r, \cS_g)\leq {r+1\choose 2}$ from \cite{H}.  We now ask what is the actual dimension of the strata $\cS_g^r$? Using hyperelliptic curves one can observe that $\cS_g^{[\frac{g-1}{2}]}\neq \emptyset$
even though the expected dimension of this locus as a determinantal variety is very negative. Moreover, the locus $\cS_{3r}^r$ is non-empty and consists of (theta characteristics on) curves $C\subset \PP^r$ which are extremal from the point of view of Castelnuovo's bound. Therefore one cannot hope that the dimension of $\cS_g^r$ be always $3g-3-{r+1\choose 2}$. However this should be the case, and the locus $\cS_g^r$ should enjoy certain regularity properties, when $r$ is relatively small with respect to $g$. We recall the following precise prediction from \cite{F1}:
\begin{conjecture}\label{strat}
For $r\geq 1$ and $g\geq {r+2\choose 2}$, there exists a component of the locus $\cS_g^r$ having codimension ${r+1\choose 2}$ inside $\cS_g$.
\end{conjecture}
The conjecture is proved in \cite{F1} for all integers $1\leq r\leq 9$ and $r=11$. We point out that $\cS_g^1$ coincides with the divisor $\Theta_{\mathrm{null}}$ studied in \cite{T2} and \cite{F2}. To prove Conjecture \ref{strat} it suffices to exhibit a single spin curve $[C, \theta]\in \cS_g^r$ with an injective Gaussian map $\psi_{\theta}$. As further evidence, we mention the following result, see \cite{F1} Proposition 2.4:
\begin{theorem}
We fix $g, r\geq 1$. If $\cS_{g-1}^r$ has a component of codimension ${r+1\choose 2}$ inside $\cS_{g-1}$ then $\cS_g^r$ has a component of codimension ${r+1\choose 2}$ inside $\cS_g$.
\end{theorem}
One could ask whether in the range $g\geq {r+2\choose 2}$, the locus $\cS_g^r$ is pure-dimensional, or even irreducible. Not much evidence in favor of this speculation exists, but there are no counterexamples either. We mention however that $\cS_g^2$ is pure of codimension $3$ in $\cS_g^-$, and when the locus $\cS_g^3$ has pure codimension $6$ in $\cS_g^+$ for $g\geq 8$, see \cite{T1}.

\subsection{Syzygies of theta characteristics and canonical rings of surfaces.} For a spin curve $[C, \theta]\in \cS_g$ one can form the graded ring
of global sections
$$R(C, \theta):=\bigoplus_{n=0}^{\infty} H^0(C, \theta^{\otimes n}).$$
Note that the canonical ring $R(C, K_C)$ appears as a graded subring of $R(C, \theta)$.
\begin{question} For a general $[C, \theta]\in \cS_g$ (or in $\cS_g^r$), describe the syzygies of $R(C, \theta)$.
\end{question}
Some tentative steps in this direction appear in \cite{R}, where it is shown that with a few exceptions, $R(C, \theta)$ is generated in degree at most $3$. The interest in this question comes to a large extent from the study of surfaces of general type. Let $S\subset \PP^{r+1}$ be a canonically embedded surface of general type and assume for simplicity that $H^1(S, \OO_S)=0$ and $r=p_g(S)-1\geq 3$. Then a general hyperplane section $C\in |K_S|$ comes equipped with an $r$-dimensional theta characteristics, that is, $[C, \theta:=\OO_C(1)]\in \cS_g^r$. By restriction there is a surjective morphism of graded rings
$R(S, K_S)\rightarrow R(C, \theta)$ and the syzygies of the two rings are identical, that is,
$$K_{p, q}(S, K_S)\cong K_{p, q}(C, \theta), \ \mbox{ for all } p, q\geq 0.$$
It is worth mentioning that using Green's duality theory \cite{G}, one finds the isomorphism
$$K_{p-1, 2}(C, \theta)^{\vee}\cong K_{r-2, 2}(C, \theta)\ \mbox{ and } K_{p-2, 3}(C, \theta)^{\vee}\cong K_{r-p+1, 1}(C, \theta)$$
between the various Koszul cohomology groups.
In a departure from the much studied case of syzygies of canonical curves,  the graded Betti diagram of a theta characteristic has three
non-trivial rows.  For two-torsion points $\eta\in J_2(C)$ a precise \emph{Prym-Green} conjecture concerning the groups $K_{p, q}(C, K_C\otimes \eta)$ has been formulated (and proven for bounded genus) in \cite{FL}. There is no clear prediction yet for the vanishing of $K_{p, q}(C, \theta)$.

\subsection{The Scorza correspondence on the moduli space of even spin curves.} To a non-effective even theta characteristic $[C, \theta]\in \cS_g^+$ one can associate the \emph{Scorza correspondence}
$$R_{\theta}:=\bigl\{(x, y)\in C\times C: H^0(C, \theta(x-y))\neq 0\bigr\}.$$
Denoting by $\pi_1, \pi_2: C\times C\rightarrow C$ the two projections and by $\Delta\subset C\times C$ the diagonal, the cohomology class of the Scorza curve can be computed:
$$\OO_{C\times C}(R_{\theta})= \pi_1^*(\theta)\otimes \pi_2^*(\theta)\otimes \OO_{C\times C}(\Delta).$$ By the adjunction formula, $p_a(R_{\theta})=1+3g(g-1)$. The curve $R_{\theta}$, first considered by Scorza \cite{Sc}, reappears in the modern literature in the beautiful paper \cite{DK}, where it plays an important role in the construction of an explicit birational isomorphism between $\cM_3$ and $\cS_3^+$. It is shown in \cite{FV1} that $R_{\theta}$ is smooth for a general even spin curve, hence one can consider the Scorza map at the level of moduli space, that is,
$$\mathrm{Sc}:\ss_g^+\dashrightarrow \mm_{1+3g(g-1)}, \ \ \mbox{ Sc}[C, \theta]:=[R_{\theta}].$$
Since $\ss_g^+$ is a normal variety, the rational map $\mathrm{Sc}$ extends to a regular morphism outside a closed set of $\ss_g^+$ of codimension at least two. It is of interest to study this map, in particular to answer the following questions:

\noindent (1) What happens to the map $\mathrm{Sc}$ over the general point of the boundary divisor $\thet$, when the determinantal definition of $R_{\theta}$ breaks down?

\noindent (2) What are the degenerate Scorza curves corresponding to general points of the boundary divisors $A_i, B_i\subset \ss_g^+$ for $i=0, \ldots, [\frac{g}{2}]$?

\noindent (3) Understand the Scorza map at the level of divisors, that is, find a complete description of the homomorphism
$$\mathrm{Sc}^*:\mathrm{Pic}(\mm_{1+3g(g-1)})\rightarrow \mathrm{Pic}(\ss_g^+).$$
Answers to all these questions are provided in the forthcoming paper \cite{FI}.

\vskip 3pt
To give one example, we explain one of the results proved. For a general point $[C, \theta]\in \thet$, we denote by $\Sigma_{\theta}$ the \emph{trace curve} induced by the pencil $\theta\in W^1_{g-1}(C)$, that is, $$\Sigma_{\theta}:=\{(x, y)\in C\times C: H^0(C, \theta(-x-y))\neq 0\}.$$  We set
$\delta:=\Sigma_{\theta}\cap \Delta$ and it is easy to see that for a generic choice of $[C, \theta]\in \thet$, the set $\delta$ consists of
$4g-4$ distinct points.
\vskip 3pt

 We consider a family $\{(C_t, \theta_t)\}_{t\in T}$  of even theta characteristics over a $1$-dimensional base, such that for a point $t_0\in T$ we have that $[C_{t_0}, \theta_{t_0}]=[C, \theta]\in \thet$ and $h^0(C_t, \theta_y)=0$ for $t\in T_0:=T-\{t_0\}$. In particular, the cycle $R_{\theta_t} \subset C_t\times C_t$ is defined for $t\in T_0$. We prove the following result:
\begin{theorem}
The flat limit of the family of Scorza curves $\{R_{\theta_t}\}_{t\in T_0}$ corresponding
to $t=t_0$,  is the non-reduced cycle $$\Sigma_{\theta}+2\Delta\subset C\times C.$$ The associated stable curve $\mathrm{Sc}[C, \theta]\in \mm_{1+3g(g-1)}$ can be described as the transverse union
$\Sigma_{\theta}\cup_{\delta} \widetilde{\Delta}$,
where $\widetilde{\Delta}$ is the double cover of $\Delta$ branched over $\delta$.
\end{theorem}
A proof of this result for $g=3$ using theta functions is given in \cite{GSM}.


\begin{thebibliography}{EMS}
\bibitem[ACGH]{ACGH} E. Arbarello, M. Cornalba, P. A. Griffiths, J. Harris,
{\em{Geometry of algebraic curves}}, Grundlehren der mathematischen Wissenschaften 267  (1985), Springer Verlag.
\bibitem[Arf]{Arf} C. Arf, {\em{Untersuchungen \"uber quadratische Formen in K\"orpern der Charakteristik $2$}}, J. reine angewandte Mathematik
\textbf{183} (1941), 148-167.
\bibitem[At]{At} M. Atiyah, {\em{Riemann surfaces and spin structures}}, Annales Scient. \'Ecole Normale Sup. \textbf{4} (1971), 47-62.
\bibitem[Ba]{Ba} H.F. Baker, {\em{Abelian functions: Abel's theorem and the allied theory of theta functions}},  1897, reissued in the Cambridge Mathematical Library 2006.
\bibitem[BDPP]{BDPP} S. Boucksom, J.P. Demailly, M. P\u{a}un and T. Peternell, {\em{The pseudo-effective cone of a compact K\"ahler manifold and varieties of negative Kodaira dimension}}, arXiv:math/0405285.
\bibitem[BL]{BL} C. Birkenhake and H. Lange, {\em{Complex abelian varieties}}, Grundlehren der mathematischen Wissenschaften  302, 2nd Edition 2004, Springer Verlag.
\bibitem[CDvG]{CDvG} S. Cacciatori, F. Dalla Piazza and B. van Geemen, {\em{Modular forms and three loop superstring amplitudes}}, Nuclear Physics B \textbf{800} (2008), 565-590.
\bibitem[Ca]{Ca} L. Caporaso, {\em{A compactification of the universal Picard variety over the moduli space of stable curves}}, Journal of the American Mathematical Society \textbf{7} (1994), 589-660.
\bibitem[CC]{CC} L. Caporaso and C. Casagrande {\em{Combinatorial properties of stable spin curves}}, Communications in Algebra \textbf{359} (2007), 3733-3768.
\bibitem[CS]{CS} L. Caporaso and E. Sernesi, {\em{Characterizing curves by their theta characteristics}}, J. reine angewandte Mathematik
\textbf{562} (2003), 101-135.
\bibitem[Cob]{Cob} A. Coble, {\em{Algebraic geometry and theta functions}}, American Mathematical Society Colloquium Publications, Volume X 1929, Reprinted and revised in 1961.
\bibitem[Cor]{C} M. Cornalba, {\em{Moduli of curves and theta characteristics}},
in: Lectures on Riemann surfaces (Trieste, 1987), World Sci. Publ., 560-589.
\bibitem[DO]{DO} I. Dolgachev and D. Ortland, {\em{Point sets in projective space and theta functions}}, Ast\'erisque, Volume 165  (1988), Soci\'et\'e Math\'ematique de France.
\bibitem[DK]{DK} I. Dolgachev and V. Kanev, {\em{Polar covariants of plane cubics and quartics}}, Advances in Mathematics \textbf{98} (1993), 216-301.
\bibitem[EH]{EH} D. Eisenbud and J. Harris, \ {\em{The Kodaira dimension  of the
moduli space of curves of genus $\geq 23$ }}, Inventiones Math.
\textbf{90} (1987), 359--387.
\bibitem[F1]{F1} G. Farkas, {\em{Gaussian maps, Gieseker-Petri loci
and large theta characteristics}}, J. reine angewandte Mathematik, \textbf{581} (2005), 151-173.
\bibitem[F2]{F2} G. Farkas, {\em{The birational type of the moduli space of even spin curves}}, Advances in Mathematics \textbf{223} (2010), 433-443.
\bibitem[F3]{F3} G. Farkas, {\em{Brill-Noether geometry on moduli spaces of spin curves}}, in: Classification of algebraic varieties, Schiermonnikoog 2009 (C. Faber, G. van der Geer, E. Looijenga editors), Series of Congress Reports, European Mathematical Society 2011, 259-276.
\bibitem[F4]{F4} G. Farkas, {\em{Koszul divisors on moduli spaces of
curves}}, American Journal of Mathematics \textbf{131} (2009), 819-869.
\bibitem[FI]{FI} G. Farkas and E. Izadi, {\em{The Scorza correspondence on the moduli space of stable spin curves}}, in preparation.
\bibitem[FL]{FL} G. Farkas and K. Ludwig, {\em{The Kodaira dimension of the moduli
space of Prym varieties}}, Journal of the European Mathematical Society \textbf{12} (2010), 755-795.
\bibitem[FV1]{FV1} G. Farkas and A. Verra, {\em{The geometry of the moduli space of odd spin curves}}, arXiv:1004.0278.
\bibitem[FV2]{FV2} G. Farkas and A. Verra, {\em{Moduli of theta characteristics via Nikulin  surfaces}},  arXiv:1104.0273, to appear in Mathematische Annalen.
\bibitem[Fed]{Fed} M. Fedorchuk, {\em{The log canonical model program on the moduli space of stable curves of genus four}}, arXiv:1106.5012, to appear in International Mathematical Research Notices.
\bibitem[Fr1]{Fr1} G. Frobenius, {\em{\"Uber Thetafunktionen mehrerer Variablen}}, J. reine angewandte Mathematik \textbf{96} (1884), 100-122.
\bibitem[Fr2]{Fr2} G. Frobenius, {\em{\"Uber Gruppen von Thetacharakteristiken}}, J. reine angewandte Mathematik \textbf{96} (1884), 81-99.
\bibitem[Fr3]{Fr3} G. Frobenius, {\em{Antrittsrede}}, Sitzungsberichte der K\"oniglich Preussischen Akademie der Wissenschaften zu Berlin 1893,
626-632.
\bibitem[vGS]{vGS} B. van Geemen and A. Sarti, {\em{Nikulin involutions on $K3$ surfaces}}, Mathematische Zeitschrift \textbf{255} (2007), 731-753.
\bibitem[Go]{Go} A. G\"opel, {\em{Theoriae Transcendentium Abelianarum primi ordinis adumbratio levis}}, J. reine angewandte Mathematik \textbf{35} (1877), 277-312.
\bibitem[G]{G} M. Green, {\em{Koszul cohomology and the cohomology of projective varieties}}, Journal of Differential Geometry \textbf{19} (1984), 125-171.
\bibitem[GH]{GH} B. Gross and J. Harris, {\em{On some geometric constructions related to theta characteristics}}, in: Contributions to automorphic forms, geometry and number theory, Johns Hopkins University Press, Baltimore 2004, 279-311.
\bibitem[Gr]{Gr} S. Grushevsky, {\em{Superstring scattering amplitudes in higher genus}}, Communications in Mathematical Physics \textbf{294} (2010), 343-352.
\bibitem[GSM]{GSM} S. Grushevsky and R. Salvati Manni, {\em{The Scorza correspondence in genus $3$}}, arXiv:1009.0375.
\bibitem[H]{H} J. Harris, {\em{Theta characteristics on algebraic curves}}, Transactions American Mathematical Society \textbf{271} (1982), 611-638.
\bibitem[HM]{HM} J. Harris and D. Mumford, {\em{On the Kodaira
dimension of $\mm_g$}}, Inventiones Math. \textbf{67} (1982), 23-88.
\bibitem[Kr]{Kr} A. Krazer, {\em{Lehrbuch der Thetafunktionen}}, Verlag von B.G. Teubner 1903, Leipzig.
\bibitem[Log]{Log} A. Logan, {\em{The Kodaira dimension of moduli spaces of curves
with marked points}}, American Journal of Mathematics \textbf{125} (2003),
105-138.
\bibitem[Lud]{Lud} K. Ludwig, {\em{On the geometry of the moduli space of spin
curves}},   Journal of  Algebraic Geometry \textbf{19} (2010), 133-171.
\bibitem[Na]{Na} D. Nagaraj, {\em{On the moduli of curves with theta characteristics}}, Compositio Math. \textbf{75} (1990), 287-297.
\bibitem[Ma]{Ma} A. Mattuck, {\em{Arthur Bryan Coble}}, Bulletin of the American Mathematical Society \textbf{76} (1970), 693-699.
\bibitem[M1]{M1} S. Mukai, {\em{Curves, $K3$ surfaces and Fano $3$-folds of genus $\leq 10$}}, in: Algebraic geometry and commutative algebra
357-377, 1988 Kinokuniya, Tokyo.
\bibitem[M2]{M2} S. Mukai, {\em{Curves and Grassmannians}}, in: Algebraic Geometry and Related Topics, eds. J.-H. Yang, Y. Namikawa, K. Ueno, 19-40, 1992   International Press.
\bibitem[M3]{M3} S. Mukai, {\em{Curves and symmetric spaces I}}, American Journal of Mathematics \textbf{117} (1995), 1627-1644.
\bibitem[M4]{M4} S. Mukai, {\em{Curves and symmetric spaces II}}, Annals of Mathematics \textbf{172} (2010), 1539-1558.
\bibitem[Mu]{Mu} D. Mumford, {\em{Theta characteristics of an algebraic curve}}, Annales Scient. \'Ecole Normale Sup. \textbf{4} (1971), 181-192.
\bibitem[P]{P} R. Pandharipande, {\em{A compactification over $\mm_g$ of the universal moduli space of slope-semistable vector bundles}}, Journal of the American Mathematical Society \textbf{9} (1996), 425-471.
\bibitem[Pu]{Pu} A. Putman, {\em{The Picard group of the moduli space of curves with level structure}}, to appear in Duke Mathematical Journal.
\bibitem[R]{R} M. Reid, {\em{Infinitesimal view of extending a hyperplane section-deformation theory and computer algebra}}, in: Algebraic geometry L'Aquila 1988, Lecture Notes in Mathematics \textbf{1417}, 214-286, Springer Verlag.
\bibitem[Ro]{Ro} G. Rosenhain, {\em{M\'emoire sur les fonctions de deux variables, qui sont les inverses des int\'egrales ultra-elliptiques de la premi\`ere classe}}, Paris 1851.
\bibitem[SM1]{SM} R. Salvati Manni, {\em{Modular varieties with level $2$ theta structure}}, American Journal of Mathematics \textbf{116} (1994), 1489-1511.
\bibitem[SM2]{SM2} R. Salvati Manni, {\em{Remarks on superstring amplitudes in higher genus}}, arXiv:0804.0512, Nuclear Physics B, \textbf{801} (2008), 163-173.
\bibitem[Sc]{Sc} G. Scorza, {\em{Sopra le curve canoniche di uno spazio lineare qualunque e sopra
certi loro covarianti quartici}}, Atti Accad. Reale Sci. Torino \textbf{35} (1900), 765-773.
\bibitem[T1]{T1} M. Teixidor i Bigas, {\em{Half-canonical series on algebraic curves}}, Transactions of the American Mathematical Society \textbf{302} (1987), 99-115.
\bibitem [T2]{T2} M. Teixidor i Bigas, {\em{The divisor of curves with a vanishing theta-null}}, Compositio
Math. \textbf{66} (1988), 15-22.
\bibitem[Wi]{Wi} W. Wirtinger, {\em{Untersuchungen \"uber Thetafunktionen}}, Verlag von B.G. Teubner, Leipzig 1895.

\end{thebibliography}
\end{document}